\documentclass[12pt]{amsart}
\usepackage[utf8]{inputenc}

\usepackage{amsmath}%
\usepackage{amsfonts}%
\usepackage{amssymb}%
\usepackage{amsthm}
\usepackage{bm}
\usepackage{graphicx}
\usepackage{stmaryrd}
\usepackage{centernot}
\usepackage{url}
\usepackage[english]{babel}
\usepackage{geometry}
\usepackage{enumerate}

\usepackage{tikz}
\usepackage{wrapfig}
\usetikzlibrary{positioning}

\theoremstyle{plain}
\newtheorem{theorem}{Theorem}[section]
\newtheorem{corollary}[theorem]{Corollary}
\newtheorem{lemma}[theorem]{Lemma}
\newtheorem{fact}[theorem]{Fact}
\newtheorem{claim}[theorem]{Claim}
\newtheorem{proposition}[theorem]{Proposition}

\theoremstyle{definition}
\newtheorem{definition}[theorem]{Definition}
\newtheorem{example}[theorem]{Example}
\newtheorem{remark}[theorem]{Remark}

\renewcommand{\subset}{\subseteq}

\newcommand{\forces}{\Vdash}

\mathchardef\mhyphen="2D
\renewcommand{\P}{\mathbb{P}}
\newcommand{\Q}{\mathbb{Q}}
\newcommand{\BPi}{\mathbf{\Pi}}
\newcommand{\BSigma}{\mathbf{\Sigma}}

\newcommand{\R}{\mathbb{R}}
\newcommand{\Z}{\mathbb{Z}}
\newcommand{\N}{\mathbb{N}}

\newcommand{\set}[2]{ \left\{ #1 :\, #2 \right\} }
\newcommand{\seqq}[2]{ \left\langle #1 :\, #2\right\rangle }

\begin{document}
\title{Actions of tame abelian product groups}
\date{\today}

\author[Shaun Allison]{Shaun Allison}
\address{Department of Mathematical Sciences, Carnegie Mellon University, Pittsburgh, PA 15213}
\email{sallison@andrew.cmu.edu}

\author[Assaf Shani]{Assaf Shani}
\address{Department of Mathematics, Harvard University, Cambridge, MA 02138, USA}
\email{shani@math.harvard.edu}

\maketitle

\begin{abstract}
A Polish group $G$ is \textit{tame} if for any continuous action of $G$, the corresponding orbit equivalence relation is Borel.
When $G = \prod_n \Gamma_n$ for countable abelian $\Gamma_n$, Solecki \cite{Solecki1995} gave a characterization for when $G$ is tame. 
In \cite{DingGao2017}, Ding and Gao showed that for such $G$, the orbit equivalence relation must in fact be potentially $\BPi^0_6$, while conjecturing that the optimal bound could be $\BPi^0_3$.
We show that the optimal bound is $D(\BPi^0_5)$ by constructing an action of such a group $G$ which is not potentially $\mathbf{\Pi}^0_5$, and show how to modify the analysis of \cite{DingGao2017} to get this slightly better upper bound.
It follows, using the results of Hjorth, Kechris, and Louvaeu \cite{HKL1998}, that this is the optimal bound for the potential complexity of actions of tame abelian product groups.
Our lower-bound analysis involves forcing over models of set theory where choice fails for sequences of finite sets.
\end{abstract}

\section{Introduction}\label{section: introduction}
Let $a\colon G\curvearrowright X$ be a continuous action of a Polish group $G$ on a Polish space $X$.
The \textbf{orbit equivalence relation} induced by the action is the equivalence relation $E_a$ on $X$ defined by $x\mathrel{E_a} y\iff\exists g\in G(g\cdot x=y)$.
Such an orbit equivalence relation $E_a$ is always analytic as a subset of $X\times X$.
A natural question is whether or not it is in fact Borel.

A Polish group $G$ is \textbf{tame} if for any continuous action $a\colon G\curvearrowright X$, the orbit equivalence relation $E_a\subseteq X\times X$ is Borel.
For example, if $G$ is compact or locally compact, then $E_a$ must be closed ($\mathbf{\Pi}^0_1$) or $F_\sigma$ ($\mathbf{\Sigma}^0_2$) respectively.
In general, by the universal $G$-space construction of Becker and Kechris \cite[Section 2.6]{BeckerKechris1996}, a Polish group $G$ is tame if and only if there is a single countable ordinal $\alpha$ such that $E_a$ is $\mathbf{\Pi}^0_\alpha$ for any continuous $a\colon G\curvearrowright X$.
By results of Becker and Kechris \cite{BeckerKechris1996}, $G$ is tame if and only if for any \emph{Borel} action $\alpha\colon G\curvearrowright X$, the orbit equivalence relation $E_a$ is Borel.

In \cite{Sami1994}, Sami asked if all abelian Polish groups are tame. Solecki~\cite{Solecki1995} answered in the negative, giving the following complete characterization for tameness among the \textbf{abelian product groups}, that is, the Polish groups of the form $\prod_{n < \omega} \Gamma_n$ for countable discrete abelian groups $\Gamma_n$, equipped with the product topology.

\begin{theorem}(Solecki~\cite{Solecki1995})
An abelian product group $\prod_{n\in\omega}\Gamma_n$ is tame if and only if the following conditions hold:
\begin{itemize}
    \item $\Gamma_n$ is torsion for all but finitely many $n$, and
    \item for any prime $p$, for all but finitely many $n$, $\Gamma_n$ is $p$-compact, i.e. the $p$-component of $\Gamma_n$ is of the form $F\times \mathbb{Z}(p^\infty)^k$ for some finite abelian $p$-group $F$ and $k\in\omega$.
\end{itemize}
\end{theorem}

For example, $\mathbb{Z}(p^\infty)^\omega$, $\prod_{p\textrm{ prime}}\mathbb{Z}(p^\infty)$, $(\bigoplus_{p\textrm{ prime}} \mathbb{Z}_p)^\omega$, and $\prod_{p\textrm{ prime}} \mathbb{Z}_p^{<\omega}$ are tame, while $\mathbb{Z}^\omega$ and
$(\mathbb{Z}_p^{<\omega})^\omega$ are not tame.
Here, for $p$ prime, $\mathbb{Z}(p^\infty)$ denotes the quasicyclic $p$-group $\mathbb{Z}(p^\infty)\simeq\set{z\in\mathbb{C}}{\exists n (z^{p^n}=1)}$, while $\mathbb{Z}_p^{<\omega}$ denotes the group
$\bigoplus_{n\in\omega}\mathbb{Z}_p$.

By applying finer notions of complexity to the induced orbit equivalence relations, we can analyze finer notions of tameness for Polish groups. 
Let $\mathbf{\Gamma}$ be a class of sets in Polish spaces, which is closed under continuous preimages. An equivalence relation $E$ on a Polish space $X$ is said to be \textbf{potentially} $\Gamma$ if there is a Polish topology $\tau$ on $X$, generating the same Borel structure, such that $E\subset X\times X$ is in $\mathbf{\Gamma}$ with respect to the product topology $\tau\times\tau$.

The notion of potential complexity reflects the inherent complexity of an equivalence relation, as it is respected by Borel reductions. Given two equivalence relations $E$ and $F$ on Polish spaces $X$ and $Y$ respectively, a map $f\colon X\to Y$ is a \textbf{reduction} from $E$ to $F$ if $x\mathrel{E}y\iff f(x)\mathrel{F}f(y)$ holds for all $x,y\in X$.
We say that $E$ is \textbf{Borel reducible} to $F$, denoted $E\leq_B F$, if there is a Borel map $f\colon X\to Y$ reducing $E$ to $F$.
It is an easy fact that an equivalence relation $E$ is potentially $\mathbf{\Gamma}$ if and only if there is an equivalence relation $F$ on a Polish space $Y$ such that $E$ is Borel reducible to $F$ and $F$ is in $\mathbf{\Gamma}$, as a subset of $Y\times Y$.
For more background on potential Borel complexity, including a proof of this fact, see \cite{HKL1998}.

We then say that a Polish group $G$ is $\mathbf{\Gamma}$-\textbf{tame} if for every Polish space $X$ and continuous action $a : G \curvearrowright X$, the induced orbit equivalence relation $E_a$ is potentially $\mathbf{\Gamma}$. 
For example, restating earlier remarks in this language, the compact Polish groups are $\BPi^0_1$-tame, while the locally-compact Polish groups are $\BSigma^0_2$-tame.


Solecki's proof shows in fact that there is a countable ordinal $\alpha$ such that for any abelian product group $\prod_n\Gamma_n$, if it is tame then it is actually $\BPi^0_\alpha$-tame.
Ding and Gao proved that $\alpha$ can be taken to be 6, that is, every tame abelian product group is in fact $\BPi^0_6$-tame \cite[Theorem 1.2]{DingGao2017}. Noting that all known examples of orbit equivalence relations induced by actions of such groups are in fact potentially $\mathbf{\Pi}^0_3$, they included the ``bold conjecture'' that their bound could be improved to $\mathbf{\Pi}^0_3$ (see \cite[Conjecture 8.3]{DingGao2017}).

We refute this conjecture and compute the optimal bound to be $D(\mathbf{\Pi}^0_5)$.
\begin{theorem}\label{thm;main}
There is a tame abelian product group which is not $\mathbf{\Pi}^0_5$-tame, and furthermore every tame abelian product group is $D(\mathbf{\Pi}^0_5)$-tame.
\end{theorem}
The majorty of this paper is dedicated towards proving the lower bound in Section~\ref{sec;lower-bounds}. In Section~\ref{section: upper bounds}, we show how to adjust the analysis of Ding-Gao to get the upper bound.

Say that $\boldsymbol{\Gamma}$ is the \textbf{(exact) potential complexity} of $E$ if $E$ is potentially $\boldsymbol{\Gamma}$ but $E$ is not potentially $\tilde{\boldsymbol{\Gamma}}$, where $\tilde{\boldsymbol{\Gamma}}$ is the class of complements of sets in $\boldsymbol{\Gamma}$.
Hjorth, Kechris, and Louveau \cite{HKL1998} completely classified the possible pointclasses $\boldsymbol{\Gamma}$ which can be realized as the exact potential complexity of an equivalence relation induced by an action of a closed subgroup of $S_\infty$.
For example, they show that among the finite Borel ranks the only possibilities are 
 $\mathbf{\Delta}^0_1$, $\mathbf{\Pi}^0_1$ $\mathbf{\Sigma}^0_2$, $\mathbf{\Pi}^0_n$, $D(\mathbf{\Pi}^0_n)$, for $n\geq 3$.
It follows that the bounds in Theorem~\ref{thm;main} are optimal. 

In each possible potential class, Hjorth, Kechris, and Louvaeu found a maximal equivalence relation with this potential complexity.
We mention below these maximal equivalence relations for the classes $\mathbf{\Pi}^0_3$, $\mathbf{\Pi}^0_4$, and $\mathbf{\Pi}^0_5$, as these will be used later.

\begin{theorem}[Hjorth-Kechris-Louveau \cite{HKL1998}]\label{thm: HKL max}
Suppose $E$ is an equivalence relation which is induced by a action of a closed subgroup of $S_\infty$. 
\begin{enumerate}
    \item $E$ is potentially $\mathbf{\Pi}^0_3$ if and only if $E\leq_B =_\R^+$;
    \item $E$ is potentially $\mathbf{\Pi}^0_4$ if and only if $E\leq_B =_\R^{++}$;
    \item $E$ is potentially $\mathbf{\Pi}^0_5$ if and only if $E\leq_B =_\R^{+++}$.
\end{enumerate}
\end{theorem}

Here, $=_\R^+$, $=_\R^{++}$, and $=_\R^{+++}$ are iterates of the Friedman-Stanley jump of equality. In general, if $E$ is an equivalence relation on a standard Borel space $X$ we define $E^+$, the \textbf{Friedman-Stanley jump} of $E$, on $X^\N$ by 
\[(x_i)_{i\in \N} \mathrel{E}^+ (y_i)_{i\in \N} \iff \set{[x_i]_E}{i\in \N} = \set{[y_i]_E}{i\in \N}.\]
For example, if $=_\R$ is the equality relation on the reals, then $=_\R^+$ is the equivalence relation on $\R^\N$ identifying two sequences if they enumerate the same set of reals. Similarly, $(=_\R^+)^+$, which we also write as $=_\R^{++}$, can be viewed is ``equality for (hereditarily countable) sets of sets of reals'', and so on.

In particular, the lower bound in Theorem~\ref{thm;main} is equivalent to showing that there is a tame abelian product group action inducing an equivalence relation which is not Borel reducible to $=_\R^{+++}$.

\subsection{The $\Gamma$-jumps of Clemens and Coskey}\label{sec: intro CC jumps}

Recently, Clemens and Coskey \cite{ClemensCoskey} defined new ``gentle'' jump operators, as follows. For a countable group $\Gamma$, and an equivalence relation $E$ on a Polish space $X$, Clemens and Coskey define the \textbf{$\Gamma$-jump of $E$}, denoted $E^{[\Gamma]}$, on $X^\Gamma$ by $
    x \mathrel{E^{[\Gamma]}} y \iff (\exists \gamma\in\Gamma) (\forall \alpha\in\Gamma) x(\gamma^{-1}\alpha) \mathrel{E} y(\alpha).$
    
The $\Gamma$-jumps generalize the usual shift actions of countable groups. For example, the $\mathbb{Z}$-jump of $=_{\{0,1\}}$ is $E_0$ and the $\mathbb{F}_2$-jump of $=_{\{0,1\}}$ is $E_\infty$, where $=_{\{0,1\}}$ is the equality relation on $\{0,1\}$.

Clemens and Coskey show that for many groups $\Gamma$, for example free groups, the $\Gamma$-jump is a jump operator on Borel equivalence relation. Furthermore, they show that the transfinite countable iterated $\Gamma$-jumps of $=_{\{0,1\}}$ are cofinal among all Borel equivalence relation which are induced by actions of the (full support) wreath product $\Gamma\wr\Gamma$.

Precisely ``how fast'' the complexity of the $\Gamma$-jumps increases is still open (see \cite{ClemensCoskey} for some upper and lower bounds). For example, they show that for any countable Borel equivalence relation $E$ on $X$, its $\mathbb{Z}$-jump is Borel reducible to $=^+$, and therefore is potentially $\mathbf{\Pi}^0_3$. Note that in this case $E$ is $\mathbf{\Sigma}^0_2$ and so the definition of $E^{[\Gamma]}$ above is naturally $\mathbf{\Sigma^0_4}$.
Furthermore, for any countable group $\Gamma$ there is a comeager subset $C\subset X^\Gamma$ such that $E^{[\Gamma]}\restriction C$ is Borel reducible to $=^+$ (see the proof of \cite[Theorem 3.5]{ClemensCoskey}). 
In fact, for varying groups $\Gamma$, the equivalence relations $E^{[\Gamma]}$, even when restricted to a comeager set, can be very different (see \cite{Shani2019}).

Clemens and Coskey asked if $E^{[\Gamma]}$ is Borel reducible to $=^+$ for any countable group $\Gamma$ and countable Borel equivalence relation $E$.
We present an unpublished result of the second author, showing that in fact the $\mathbb{Z}^2$-jump is \textit{not} Borel reducible to $=^+$ (equivalently, is not potentially $\mathbf{\Pi}^0_3$).
\begin{theorem}\label{thm: Z2 jump non pi3}
$E_0^{[\mathbb{Z}^2]}$ is not potentially $\mathbf{\Pi}^0_3$.
\end{theorem}
The key will be to restrict $E_0^{[\mathbb{Z}^2]}$ to a subset of its domain, on which the $\mathbb{Z}^2$ rows are all periodic with distinct prime periods (and this subset is meager with respect to the standard product topology on its domain).
The argument is closely related to the non-potentially $\mathbf{\Pi}^0_3$ action of $\mathbb{Z}\times\Gamma^{\omega}$ described in Section~\ref{subsec;non-pi3}.

\subsection*{Acknowledgments}
We would like to thank Clinton Conley for his support throughout this project, and for many helpful discussions. We would also like to thank John Clemens and Samuel Coskey for sharing early drafts of their paper \cite{ClemensCoskey}.

\section{Preliminaries}\label{sec;prelim}
We will assume familiarity with set theory and the method of forcing, as in \cite{Jech2003} or \cite{Kunen2011}. We will also assume familiarity with the basic theory of Polish spaces and Polish groups, as well as the Borel hierarchy (\cite{Kechris1995} and \cite{Gao2009} are standard references).

We use $\omega$ to denote the set of natural numbers $\N=0,1,2,...$. For an equivalence relation $E$ on $X$ and $x\in X$, its $E$-class is defined by $[x]_E=\set{y\in X}{x \mathrel{E} y}$. For a group $\Gamma$, write $\Gamma^{<\omega}$ for the finite support product $\bigoplus_{n<\omega}\Gamma$.

We use standard forcing notation and terminology as in \cite{Jech2003} or \cite{Kunen2011}.
Given a poset $\P$ and a forcing statement $\phi$, say that $\P\forces\phi$ if all conditions in $\P$ force $\phi$.



\subsection{Classification by countable structures}

Recall that an equivalence relation $E$ on a Polish space $X$ is \textbf{classifiable by countable structures} if it is Borel reducible to the isomorphism relation on the space of models of a countable language.
Equivalently, these are the equivalence relations which are (up to Borel reduction) induced by a continuous action of a closed subgroup of $S_\infty$ (see \cite[Section 1.5]{BeckerKechris1996}). In this paper, we will be focusing on the equivalence relations which are classifiable by countable structures.

Every such equivalence relation admits a complete classification by hereditarily-countable sets, that is, a map $x\mapsto A_x$ from $X$ to the hereditarily-countable sets such that $x \mathrel{E} y\iff A_x = A_y$, via the Scott analysis (see \cite[Chapter 12]{Gao2009}).
From the forcing point of view, the crucial property of such complete classifications $x\mapsto A_x$ is that the map is definable in an absolute way, as follows.
There is a formula $\varphi(x, y, \overline{a})$ in the language of set theory with parameters $\overline{a} \in V$, such that $\varphi(x, A, \overline{a})\iff A = A_x$, and
\begin{itemize}
    \item $\phi(x,A,\bar{a})$ defines a complete classification of $E$ in any model of ZF and
    \item $\phi(x,A,\bar{a})$ is absolute between models of ZF containing $x,A,\bar{a}$.
\end{itemize}
The second clause implies that the invariant $A_x$ does not change when moving to a further generic extension.
We will call any such classification $x\mapsto A_x$ an \textbf{absolute classification} of $E$.
It is well known that the Scott analysis satisfies this properties, see for example \cite[Lemma 2.4]{Friedman2000}, or \cite[Chapter 12.1]{Gao2009}.

For the equivalence relations discussed in this paper, there will always be a very simple absolute classification, with hereditarily countable sets as invariants, and the combinatorial structure of the invariants will be used to study the equivalence relation. Examples include the following.

\begin{example}\label{example:absolute_classifications}
The following are all examples of absolute classifications:
\begin{enumerate}
    \item The identity map $x\mapsto x$ is an absolute classification of $=_\R$, the equality relation on the reals;
    \item The map $x\mapsto \set{x(n)}{n\in\omega}$ is an absolute classification of $=^+$, with sets of reals as invariants;
    \item Similarly, $=^{++}$ and $=^{+++}$ admit absolute classifications with invariants in $\mathcal{P}\mathcal{P}(\R)$, and $\mathcal{P}\mathcal{P}\mathcal{P}(\R)$, respectively;
    \item If $\alpha\colon \Gamma\curvearrowright X$ is a continuous actions of a countable group $\Gamma$ on a Polish space $X$, then the map $x\mapsto \Gamma\cdot x=[x]_{E_\alpha}$ is an absolute complete classification of the induced orbit equivalence relation $E_\alpha$.
\end{enumerate}
\end{example}


\subsection{Symmetric models}

Given a set $A$ in some generic extension of $V$, let $V(A)$ be the minimal transitive exntesion of $V$ which contains $A$ and satisfies ZF. 
$V(A)$ can be seen as the set-theoretic definable closure of $A$ over $V$, as follows.

\begin{fact}
\label{fact;definability-V(A)}
The following holds in the model $V(A)$.
For any set $X$, there is some formula $\psi$, a sequence of parameters $\bar{a}$ from the transitive closure of $A$, and $v \in V$ such that $X$ is the unique set satisfying $\psi(X,A,\bar{a},v)$.
Equivalently, there is a formula $\varphi$ such that $X=\set{x}{\varphi(x,A,\bar{a},v)}$.
\end{fact}
This follows from the minimality of $V(A)$, see for example the arguments in \cite[p.~8]{Shani2021}.
In this case we say that $X$ is definable in $V(A)$ using $A$, $\bar{a}$ and $v$. In the special case that $\bar{a} = \emptyset$, we say that $X$ is \textbf{definable from $\mathbf{A}$ over $\mathbf{V}$}.
Of particular interest are invariants, for some equivalence relations, which are definable from $A$ over $V$.


\begin{theorem}[{\cite[Lemma 3.6]{Shani2021}}]\label{theorem:invariants_borel_reduction}
Suppose $E$ and $F$ are analytic equivalence relations on $X$ and $Y$ respectively with absolute classifications $x\mapsto A_x$ and $y\mapsto B_y$, and $E$ is Borel reducible to $F$ as witnessed by $f : X \rightarrow Y$.
Then for any $x \in X$ in some generic extension of $V$, $B_{f(x)}$ is definable from $A_x$ over $V$, and furthermore $V(A_x) = V(B_{f(x)})$.
\end{theorem}
Let $A=A_x$ and $B=B_{f(x)}$.
The idea is that $f$ induces an injective map between $E$-invariants and $F$-invariants, definable in an absolute manner in any generic extension. So $A$ and $B$ are definable from one another via this map.

Note that there will usually be no $x\in V(A)$ such that $A=A_x$.
We will still refer to such set $A$ as an $E$-invariant (or absolute $E$-invariant), as it is of the for $A_x$ for some $x$ in a further extension.
For example, any set of reals can be an invariant for $=^+$, in \textit{some} generic extension, after adding a countable enumeration of it.

\begin{corollary}
\label{cor : irred to cong}
Suppose $E$ is a Borel equivalence relation on a standard Borel space \( X \), and $x\mapsto A_x$ is an absolute classification of $E$.
Let \(x\) be an element of \( X \) in some generic extension of \( V \), and set $A=A_x$.
\begin{enumerate}
    \item If $E\leq_B {=^+_\R}$ then there is a set of reals $B\in V(A)$ such that $B$ is definable from $A$ over $V$ and $V(A)=V(B)$.
    \item If $E\leq_B =_\R^{++}$ then there is a set of sets of reals $B\in V(A)$ such that $B$ is definable from $A$ over $V$ and $V(A)=V(B)$.
    \item If $E\leq_B =_\R^{+++}$ then there is a set of sets of sets of reals $B\in V(A)$ such that $B$ is definable from $A$ over $V$ and $V(A)=V(B)$.    
\end{enumerate}
\end{corollary}
\begin{proof}
Part (1) follows from Theorem~\ref{theorem:invariants_borel_reduction}, since $=_\R^+$ admits an absolute classification with sets of reals as invariants. Similarly for parts (2) and (3).
\end{proof}
The above is used to prove Borel irreducibility results. For example, to prove that $E$ is not Borel reducible to $=^+$, we will look for a ``sufficiently complicatd'' invariant $A$ for $E$ such that the model $V(A)$ cannot be generated by any set of reals.

The invariant $A$ will often be $A_x$ where $x$ is a Cohen-generic real, in this case we get that the irreducibility persists when restricting $E$ to comeager sets (see \cite[Theorem 3.8]{Shani2021}).
This can also be used to study Borel homomorphisms on comeager sets. For example:
\begin{lemma}[{\cite[Lemma 2.5]{Shani2019}}]\label{lem;generic-erg-symm-model}
Suppose $E$ and $F$ are Borel equivalence relations on $X$ and $Y$ respectively and $x\mapsto A_x$ and $y\mapsto B_y$ are absolute classifications of $E$ and $F$ respectively.
The following are equivalent.
\begin{enumerate}
    \item For every partial Borel homomorphism  $f\colon E\rightarrow_B F$, defined on some non-meager set, $f$ maps a non-meager set into a single $F$-class;
    \item If $x\in X$ is Cohen-generic over $V$ and $B$ is an $F$-invariant in $V(A_x)$ which is definable from $A_x$ over $V$, then $B\in V$.
\end{enumerate}
\end{lemma}
\begin{remark}
Suppose $E$ is generically ergodic. Then condition (1) in Lemma~\ref{lem;generic-erg-symm-model} is equivalent to: 
$E$ is generically $F$-ergodic. That is, every partial Borel homomorphism from $E$ to $F$, defined on a comeager set, sends a comeager set into a single $F$-class.
\end{remark}

Given a set $A$ appearing in a generic extension of $V$, say that two elements $x_1, x_2 \in V(A)$ \textbf{have the same type over} $\mathbf{V, A}$ if for any formula $\varphi$ in the language of set theory and tuple $v$ of parameters in $V$, we have $V(A) \models \varphi(x_1, A, v) \leftrightarrow \varphi(x_2, A, v)$.
In other words, $x_1,x_2$ are indiscernible over $V,A$.

\section{Motivation}\label{sec: motivation}
Towards proving the lower bound of Theorem~\ref{thm;main}, we must define some ``sufficiently complicated'' actions of our product groups. This section is dedicated to motivating the definitions of these actions, and explain how these definitions are tailored to provide the desired irreducibility results.

With the point of view of Theorem~\ref{theorem:invariants_borel_reduction}, we are set to search for actions with ``sufficiently complicated'' orbits, in terms of set theoretic definability.
We explain first how the (very) well known irreduciblity ${E_0} \not\leq_B {=_\R}$ is presented via the point of view of definability of invariants.
While this is simply the well known ergodicity proof, the rest of the irreducibility proofs in this paper will in fact follow precisely the same argument, presented in a generalized context.

\subsection{Ergodicity and weakly homogenous forcing}
Recall the equivalence relation $E_0$ on $2^\omega$, defined by $x\mathrel{E_0} y$ if and only if $x(n)$ and $y(n)$ agree for all but finitely many values of $n$.
In other words, $E_0$ is the orbit equivalence relation of the action of $\mathbb{Z}_2^{<\omega}$ on $2^\omega$.
The map $x\mapsto \mathbb{Z}_2^{<\omega}\cdot x=[x]_{E_0}$, sending $x$ to its orbit, is an absolute complete classification of $E_0$.

Let $x\in 2^\omega$ be Cohen generic over $V$.
That is, let $\P$ be the Cohen poset of finite partial functions from $\omega$ to $2$, ordered by reverse extension. Let $G\subset\P$ be a generic filter over $V$, and $x\in 2^\omega$ the real associated with the function $\bigcup G$.
\begin{claim}[see Lemma~\ref{lem;weak-hom}]\label{claim;levy}
Let $A=[x]_{E_0}=\set{y\in 2^\omega}{y\mathrel{E_0}x}$.
If $B\in V(A)$ is definable from $A$ over $V$, and $B\subseteq V$, then $B\in V$.
\end{claim}
We will identify real numbers with subsets of $\omega$, and thus elements of $2^\omega$, so that the following makes sense:
\begin{corollary}
In particular, if $B\in V(A)$ is a real, definable from $A$ over $V$, then $B\in V$, and therefore $V(B)\neq V(A)$. 
Since the reals are complete invariants for $=_\R$, it follows from Theorem~\ref{theorem:invariants_borel_reduction} that $E_0$ is not Borel reducible to $=_\R$. In fact, the previous claim is equivalent to the generic ergodicity of the $\mathbb{Z}_2^{<\omega}$ action, by Lemma~\ref{lem;generic-erg-symm-model}.
\end{corollary}
Claim~\ref{claim;levy} goes back to Levy (see \cite{Kanamori2006}), who showed that after adding a generic Cohen real $x$ to $L$, HOD calculated in $L[x]$ is $L$. In more modern terms: a weakly homogenous forcing adds no new HOD sets. 
We generalize this as follows.

\begin{definition}\label{defn: poset ergodicity}
Let $\mathbb{P}$ be a poset, $\Gamma$ a group and $a\colon \Gamma\curvearrowright\mathbb{P}$ an action of $\Gamma$ on $\mathbb{P}$ by automorphisms.
Say that the action is \textbf{ergodic} if for any $p,q\in\mathbb{P}$ there is $\gamma\in\Gamma$ such that $\gamma\cdot p$ is compatible with $q$.
\end{definition}
Recall that $\P$ is \textbf{weakly homogenous} if for any $p,q\in\P$ there is some automorphism of $\P$ sending $p$ to be compatible with $q$. That is, the entire group of automorphisms of $\P$ acts ergodically on $\P$.

The following lemmas generalize the well-known facts for weakly homogeneous posets to ergodic actions on posets, with essentially the same proofs.
Note that both lemmas are proved in ZF set theory, with no use of choice.

\begin{lemma}\label{lem;gen-orbit-type}
Suppose that $\Gamma \curvearrowright \P$ is ergodic and $x$ is $\mathbb{P}$-generic over $V$. For any $v\in V$ and formula $\varphi$ in the language of set theory, $\varphi(\Gamma\cdot x,v)$ holds in $V[x]$ if and only if $\mathbb{P}\forces \varphi(\Gamma\cdot\dot{x},\check{v})$.
\end{lemma}
In other words, the type of the generic orbit $\Gamma\cdot x$ over $V$ is decided in $V$.
In particular, any two generic orbits $\Gamma\cdot x$ and $\Gamma\cdot y$ have the same type over $V$.
\begin{proof}
It suffices to show that if $p\forces \varphi(\Gamma\cdot\dot{x},v)$ then the set of $q\in\mathbb{P}$ forcing $\varphi(\Gamma\cdot\dot{x},v)$ is predense in $\mathbb{P}$.
Indeed, given any $r\in\mathbb{P}$, by ergodicity there is some $\gamma\in\Gamma$ such that $\gamma\cdot p$ and $r$ are compatible.
Since $\gamma$ is an automorphism of $\mathbb{P}$ which fixes the name $\Gamma\cdot\dot{x}$, it follows that $\gamma\cdot p\forces\varphi(\Gamma\cdot\dot{x},v)$ (see for example \cite[Lemma 14.37]{Jech2003}).
Let $q$ be a common extension of $r$ and $\gamma\cdot p$, then $q$ is as desired.
\end{proof}

\begin{lemma}\label{lem;weak-hom}
Suppose that $\Gamma\curvearrowright \mathbb{P}$ is ergodic, and $x$ is $\mathbb{P}$-generic over $V$.
Then any set $B\subseteq V$ in $V[x]$ which is definable from the orbit $\Gamma\cdot x$ and parameters in $V$, must in fact be in $V$.
\end{lemma}

\begin{proof}
Let $A=\Gamma\cdot x$, and note that $V[x]=V(A)$.
Suppose $B$ is as above, fix a formula $\varphi$ and parameter $u\in V$ such that, in $V(A)$, $v\in B\iff \varphi(v,A,u)$.
Applying Lemma~\ref{lem;gen-orbit-type} to the statement $\varphi^{V(A)}(v,A,u)$, we conclude that $B$ can be defined in $V$ as the set of $v$ such that $\mathbb{P}\forces \varphi(\check{v},\dot{A},\check{u})$. 
\end{proof}




\begin{example}
The notion of ergodicity in Definition~\ref{defn: poset ergodicity} directly generalizes the usual, either measure theoretic or category theoretic, notions of ergodicity.
\begin{enumerate}
    \item Suppose $X$ is a Polish space and $\Gamma$ acts on $X$ by homeomorphisms. Let $\mathbb{P}$ be Cohen forcing on $X$ (forcing with open subsets), then $\Gamma$ extends to an action on $\mathbb{P}$.
    Furthermore, the action of $\Gamma$ on $X$ is generically ergodic if and only if the corresponding action of $\Gamma$ on $\mathbb{P}$ is ergodic.
    \item Suppose $(X,\nu)$ is a standard measure space and $\Gamma$ acts on $X$ by measure preserving transformations. Let $\mathbb{P}$ be Random real forcing on $X$, then $\Gamma$ extends to an action on $\mathbb{P}$. Furthermore, the action of $\Gamma$ on $X$ is ergodic if and only if the corresponding action of $\Gamma$ on $\mathbb{P}$ is ergodic.
\end{enumerate}
\end{example}

In many applications the weak homogeneity of a poset $\P$ is realized by an ergodic action of a countable group on $\P$. In the result of Levy mentioned above, it is the action of $\mathbb{Z}_2^{<\omega}$ on the Cohen poset $\P$, arising by (1) above from the action of $\mathbb{Z}_2^{<\omega}$ on $2^\omega$. 
In the study of the so called ``basic Cohen model'' (see \cite{Jech2003,Kanamori2008}), it is the group of all finite permutations of $\omega$, acting by permuting the index set of a finite support iteration of $\P$.

The main point of Definition~\ref{defn: poset ergodicity} is to allow for the usual ergodicity arguments, for example, that an ergodic action is not concretely classifiable, in the context of a countable group acting on a \textit{non-Polish} space, such as a forcing poset in some choiceless model. 
A central example for us is the following, which will be used to find interesting actions of $\mathbb{Z}_2^{<\omega}\times\Gamma^\omega$ for an arbitrary countable group $\Gamma$.

\begin{example}\label{ex;finite-choice}
Let $\bar{A} = \seqq{A_n}{n<\omega}$ and suppose $\Sigma_n \curvearrowright A_n$ is a transitive action of a countable group $\Sigma_n$ on some set $A_n$, for each $n$.
Let $\P_{\bar{A}}$ be the poset of all finite partial choice functions in $\bar{A}$.
That is, all finite functions $p$ with $\mathrm{dom}(p)\in\omega$ and $p(i)\in A_i$ for $i\in \mathrm{dom}(p)$.
The coordinate-wise action of $\bigoplus_n\Sigma_n$ on $\mathbb{P}_{\bar{A}}$ is ergodic.
\end{example}
\begin{proof}
Note that for any two $p,q\in\P$ with the same domain there is a $\gamma\in\bigoplus_n\Sigma_n$ such that $\gamma\cdot p=q$.
For arbitrary conditions $p,q$, find extensions $p',q'$ of $p,q$ respectively, and fix $\gamma$ sending $p'$ to $q'$. Then $q'$ is a common extension of $q$ and $\gamma\cdot p$.
\end{proof}
\begin{figure}[h]
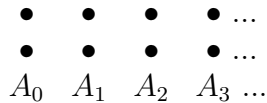

    \centering
\begin{tabular}{c c c c}
$\bullet$ & $\bullet$ & $\bullet$ & $\bullet$ ... \\

$\bullet$ & $\bullet$ & $\bullet$ & $\bullet$ ... \\
$A_0$ & $A_1$ & $A_2$ & $A_3$ ...
\end{tabular}
    \caption{Example~\ref{ex;finite-choice} for $\Sigma_n=\mathbb{Z}_2.$}
    \label{fig:pairs}
\end{figure}

The following is another interesting example of the generalized notion of ergodicity from Definition~\ref{defn: poset ergodicity}.
It will be used in Section~\ref{subsec;non-pi3} to find interesting actions of $\mathbb{Z}\times\Gamma^\omega$, for an arbitrary countable group $\Gamma$, and to conclude the results from Section~\ref{sec: intro CC jumps} about the Clemens-Coskey jumps.
\begin{example}\label{example: chinese remainder}
Let $\bar{A} = \seqq{A_p}{p \textrm{ prime}}$ and suppose $\mathbb{Z}_p\curvearrowright A_p$ acts transitively for every $p$.
We extend each action of $\mathbb{Z}_p$ to an action of $\mathbb{Z}$ via the surjective homomorphism $\mathbb{Z}\to\mathbb{Z}_p$, $k\mapsto k\mod p$.
Then the diagonal action $\mathbb{Z}\curvearrowright\mathbb{P}_{\bar{A}}$ is ergodic.
\end{example}
\begin{proof}
It suffices to show that for two conditions $r,q\in\mathbb{P}$ of the same length, there is a $k\in\mathbb{Z}$ such that $k\cdot r=q$. 
By assumption, for each $p$ in the domains of $r$ and $q$, there is some $\sigma_p\in\mathbb{Z}_p$ sending $r(p)$ to $q(p)$.
By the Chinese remainder theorem, there is some $k\in\mathbb{Z}$ such that for each $p$ in the domains of $r,q$, $k\mod p=\sigma_p$.
It follows that $k$, via the action $\mathbb{Z}\curvearrowright\mathbb{P}$, maps $r$ to $q$ .
\end{proof}
\begin{figure}[h]
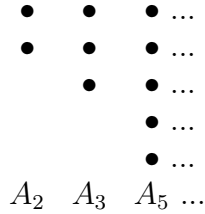

    \centering
\begin{tabular}{c c c}
$\bullet$ & $\bullet$ & $\bullet$ ... \\
$\bullet$ & $\bullet$ & $\bullet$ ... \\
 & $\bullet$ & $\bullet$ ...\\
 & & $\bullet$ ... \\
 & & $\bullet$ ... \\
$A_2$ & $A_3$ & $A_5$ ...
\end{tabular}
    \caption{Example~\ref{example: chinese remainder} with $A_p=\mathbb{Z}_p$.}
    \label{fig:primes}
\end{figure}

Below, we will use Lemma~\ref{lem;weak-hom} to prove an irreducibility to $=^+$, in a way similar to how it was used above to prove an irreducibility to $=_\mathbb{R}$.
Towards that end, we need to analyze definable sets of reals in $V(A)$, as we analyzed definable reals before. 

Since we in fact understand arbitrary definable subsets of the ground model over which we force (Lemma~\ref{lem;weak-hom}), we immediately get the desired understanding of definable sets of reals. This is as long as we force with a poset $\P$ satisfying the ergodicity assumption as above, and critically, does not add any reals. 
\begin{corollary}\label{cor;irred-to-jumps}
In Lemma~\ref{lem;weak-hom}, if $\mathbb{P}$ adds no new reals, then for any set of reals $B$, if $B$ is definable from $\Gamma\cdot x$ then $B\in V$.
\end{corollary}
\begin{proof}
This follows immediately from Lemma~\ref{lem;weak-hom}, since by assumption B is a subset of the ground model $V$.
\end{proof}

\begin{example}\label{example: choice for pairs}
Let $\bar{A}=\seqq{A_n}{n<\omega}$ be a sequence of pairs, $|A_n|=2$, and $\mathbb{P}=\P_{\bar{A}}$ the poset of finite choice functions as in Example~\ref{ex;finite-choice} (see Figure~\ref{fig:pairs}). 
Then $\mathbb{Z}_2^{<\omega}$ acts ergodically on $\mathbb{P}$.
Let $x\in\prod_n A_n$ be $\P$-generic over $V$, and consider $A=\mathbb{Z}_2^{<\omega}\cdot x$, its orbit.

Assume further that this orbit $A$ is a complete invariant for some equivalence relation $E$ (much like in Claim~\ref{lem;weak-hom} above, where $A$ was an $E_0$-class).
To get a non reducibility of $E$ to $=^+$,
we want a situation where $V(A)\neq V(B)$ for any set of reals $B$, which is definable from $A$.
This will follow at once from Corollary~\ref{cor;irred-to-jumps} assuming that forcing with $\mathbb{P}$ over $V$ does not add reals.
For the latter to hold it must be the case that $\prod_{n<\omega} A_n=\emptyset$ in the model $V$. (Otherwise, if $z\in\prod_n A_n$ is in $V$, then the set $\set{n<\omega}{z_n=x_n}$ will be a new subset of $\omega$.)
\end{example}

Such a model, having a sequence of pairs $\seqq{A_n}{n<\omega}$ with no choice function, was constructed by Cohen~\cite{Cohen1963}, based on an earlier construction of Fraenkel~\cite{Fraenkel1922} using urelements (see also \cite{Kanamori2008}).
The members of the sets $A_n$ cannot be ordinals, nor reals. Indeed, in Cohen's model each of the two members of $A_n$ is a set of reals, such that the two sets are ``sufficiently indiscernible''.

In light of  Lemma~\ref{lem;gen-orbit-type}, getting ``sufficiently indiscernible'' sets of reals is simple: given a countable group $\Gamma$ acting ergodically on a Polish space $X$, for any two generic reals $x,y\in X$, the orbits $\Gamma\cdot x$ and $\Gamma\cdot y$ satisfy the same type over $V$.

\subsection{Non potentially $\mathbf{\Pi}^0_3$ actions}\label{section: motivation non Pi3}
We now describe an action of a product group such that the orbit equivalence relation admits a generic complete invariant $A$ as in Example~\ref{example: choice for pairs}.
This construction will be revisited in Section~\ref{sec;lower-bounds}, and similar ideas will be used to produce non-potentially $\BPi^0_4$ and $\BPi^0_5$ actions.
Following the discussion above, we consider ``an ergodic action of $\mathbb{Z}_2^{<\omega}$ on pairs of $\Gamma$-orbits'', which will correspond to an action of the product group $\mathbb{Z}_2^{<\omega}\times\Gamma^\omega$.
Using Solecki's characterization mentioned in Section~\ref{section: introduction}, this group is tame for various choices of countable abelian groups $\Gamma$.
The same argument also works with $\Z_p$ instead of $\Z_2$, for any prime $p$.

Let $\Gamma$ be an infinite countable group, and consider the standard shift action of $\Gamma$ on the space $2^\Gamma$.
This action extends to a diagonal action of $\Gamma$ on the product space $X_p=(2^\Gamma)^p$ by shifting all coordinates simultanously,
\[ \gamma \cdot (x_k)_{k<p} = (\gamma \cdot x_k)_{k < p}.\]
There is also a natural action of $\mathbb{Z}_p$ on $X_p$, by rotating,
\[ \sigma \cdot (x_k)_{k<p} = (x_{k + \sigma \pmod{p}})_{k<p}. \]
These actions commute, yielding an action of $\Z_p\times\Gamma$ on $X_p$.

We can identify the group $\Z_p^{<\omega}$ as the subgroup of $\Z_p^{\omega}$ consisting of the sequences which are eventually the identity, and consider the inherited coordinatewise action $\Z_p^{<\omega} \curvearrowright X_p^\omega$ which commutes with the coordinate-wise action $\Gamma^\omega \curvearrowright X_p^\omega$. Thus we can define the action

\begin{equation*}
\alpha_p\colon\Z_p^{<\omega}\times\Gamma^\omega\curvearrowright X_p^\omega.
\end{equation*}

Fix $\alpha=\alpha_p$,
and let $E_\alpha$ be the orbit equivalence relation induced by the action $\alpha$.

\begin{proposition}\label{prop:motivation non-Pi3}
The orbit equivalence relation $E_\alpha$ is not Borel reducible to $=^+$.
\end{proposition}
It follows from Theorem~\ref{thm: HKL max} that $E_\alpha$ is not potentially $\mathbf{\Pi}^0_3$.
This already refutes Conjecture 8.3 of \cite{DingGao2017}.

The action of $\Z_p$ above can be viewed as an action of $\Z_p$ on $\Gamma$-orbits as follows.
Given $x\in X_p$ and $\sigma\in\Z_p$, let $\sigma\cdot(\Gamma\cdot x)$ be the orbit $\Gamma\cdot(\sigma\cdot x)$. Note that this is well-defined as the actions commute. 
Similarly, $\mathbb{Z}_p^{<\omega}$ acts on sequences of $\Gamma$-orbits by 
\[\sigma\cdot \seqq{\Gamma\cdot x_n}{n<\omega}=\seqq{\Gamma\cdot(\sigma_n\cdot x_n)}{n<\omega},\]
From the definition of the action we see that the map
\begin{equation*}
    x\mapsto\mathbb{Z}_p^{<\omega}\cdot\seqq{\Gamma\cdot x_n}{n<\omega}
\end{equation*} 
is an absolute classification of $E_\alpha$.
An invariant here can be thought of as a set of all finite changes of a sequence of $p$-tuples of $\Gamma$-orbits.

Let $\mathbb{Q}$ the forcing poset to add a Cohen-generic element of $X_p^\omega$, 
and force with $\Q$ over $V$ to get a generic sequence $\seqq{x_n}{n < \omega}$. Define $A_n = \Gamma \cdot x_n$ for $n < \omega$, and $\bar{A} = \seqq{\Z_p \cdot A_n}{n < \omega}$. Then $\hat{A} = \Z_{p}^{<\omega} \cdot \seqq{A_n}{n<\omega}$ is an absolute invariant for $E_{\alpha}$.
Note that $V(\seqq{A_n}{n<\omega}) = V(\hat{A})$, because $\seqq{A_n}{n<\omega}$ is in $\hat{A}$, and on the other hand, $\hat{A}$ is just the orbit of $\seqq{A_n}{n<\omega}$ under the action of $\mathbb{Z}_p^{<\omega}$.
\begin{lemma}\label{lemma: motivation non pot Pi3}
Let $\P=\P_{\bar{A}}$ be the poset of all finite choice sequences in $\bar{A}$.
\begin{enumerate}
    \item $\P$ adds no reals when forcing over $V(\bar{A})$;
    \item $x=\seqq{A_n}{n<\omega}$ is $\P$-generic over $V(\bar{A})$, and is not in $V(\bar{A})$.
\end{enumerate}
\end{lemma}
Note that each set $A_n$ has precisely $p$ distict $\Gamma$-orbits. By Example~\ref{ex;finite-choice}, the action of $\Z_p^{<\omega}$ on $\P$ is ergodic.
Following Example~\ref{example: choice for pairs}, we conclude that if $B$ is a set of reals in $V(\hat{A})$ which is definable from $\hat{A}$ over $V$, then $B$ is in $V(\bar{B})$. Since $V(\hat{B})$ is strictly larger than $V(\bar{B})$ (as the choice sequences $\seqq{A_n}{n<\omega}$ is not in $V(\bar{A})$), it follows that $V(\hat{A})$ is not of the form $V(B)$ for any set of reals $B$ which is definable from $\hat{A}$ over $V$.
By Theorem~\ref{theorem:invariants_borel_reduction}, it follows that $E_\alpha$ is not Borel reducible to $=^+$, concluding the proof of Proposition~\ref{prop:motivation non-Pi3}.

The key point behind Lemma~\ref{lemma: motivation non pot Pi3} is that the members of each $A_n$ are indiscenrible over $V$ (recall Lemma~\ref{lem;gen-orbit-type}), and therefore any two conditions in $\P$ have the same opinion about statements of the form $\check{v}\in\tau$, where $\tau$ is a name for a subset of $V$, and so any such name is decided in $V(\bar{A})$.

We prove this lemma in a more general setting in the next section (see Proposition~\ref{prop:symmetric-cohen-properties}), which will be used iteratively in the construction of more complex actions.

\section{Technical tools}
Section~\ref{section: tools actions on orbits} provides a construction of actions of abelian products groups, similar to the action of $\Z_{p}^{<\omega}\times \Gamma^\omega$ described in Section~\ref{section: motivation non Pi3} above. This construction will be iterated to find actions of higher complexity of more involved abelian product groups.

Section~\ref{section: tools indiscernibility} provides the main technical results, showing that the indiscernibility of the members of our invariants, as motivated in Section~\ref{sec: motivation}, provides the necessary analysis of the models generated by those invariants, as in Lemma~\ref{lemma: motivation non pot Pi3}. 

Section~\ref{section:tools higher rank sets} shows how reals, sets of reals, sets of sets of reals, and so on (invariants for $=_\R$, $=^+$, $=^{++}$, and so on), are stablized when iterating the constructions outlined in Section~\ref{section: tools indiscernibility}. This will be useful to prove irreducibility to the higher Friedman-Stanley jumps.

\subsection{Actions on orbits}\label{section: tools actions on orbits}
Suppose $\alpha\colon G\curvearrowright X$ and $\beta\colon\Sigma\curvearrowright X$ are commuting actions. Assume also $c\colon X\to I$ is an absolute complete classification of $\alpha$. (For example, if $G=\Gamma$ is a countable group, we may take $c(x)=\Gamma\cdot x$.)
We extend the action of $\Sigma$ to the invariants of $\alpha$ by $\sigma\cdot c(x)=c(\sigma\cdot x)$.

\begin{definition}\label{def;direct-sum-action}
Suppose $\alpha_n,\beta_n$ are commuting actions of $G_n,\Sigma_n$ on spaces $X_n$ respectively. 
Consider the natural product action $\beta_n\times\alpha_n$ of $\Sigma_n\times G_n$ on $X_n$, and the corresponding product action $\prod_n (\beta_n\times\alpha_n)$ of $\prod_n(\Sigma_n\times G_n)$ on $\prod_n X_n$.
We define the action $\bigoplus_n\beta_n\times\prod_n \alpha_n$ of $\bigoplus_n\Sigma_n\times\prod_n G_n$ on $\prod_n X_n$ as the restriction of the above action, where $\bigoplus_n\Sigma_n \times\prod_n G_n$ is identified as a subgroup of  $\prod_n(\Sigma_n\times G_n)$.
\end{definition}
Given absolute complete classifications $c_n\colon X_n\to I_n$ of the actions $G_n\curvearrowright X_n$, the map
\begin{equation*}
    x\mapsto \seqq{\Sigma_n\cdot c_n(x)}{n<\omega}
\end{equation*}
is a complete classification of the action $\prod_n \beta_n\times\prod_n\alpha_n$. Here the invariants are countable sequences whose $n$'th coordinate is the $\Sigma_n$-orbit of an invariant for the action of $G_n$.

Similarly, we get an absolute complete classification of the action $\bigoplus_n\beta_n\times\prod_n \alpha_n$:
\begin{equation*}
    x\mapsto(\bigoplus_n\Sigma_n)\cdot\seqq{c_n(x)}{n<\omega}. 
\end{equation*}
The action $\bigoplus_n\beta_n\times\prod_n \alpha_n$ will be the one increasing the potential complexity. Such action was used in Section~\ref{sec: motivation} with $G_n=\Gamma$ and $\Sigma_n=\Z_p$ for a fixed prime $p$.
Often the group $\Sigma_n$ will be finite, in which case $\Sigma_n\cdot c_n(x)=\set{c_n(\sigma(x))}{\sigma\in \Sigma_n}$ is a finite set of invariants in $I_n$.
The sequence $\seqq{c_n(x)}{n<\omega}$ can be viewed as a choice sequence through $\seqq{\Sigma_n\cdot c_n(x)}{n<\omega}$, and $(\bigoplus_n\Sigma_n)\cdot\seqq{c_n(x)}{n<\omega}$ is the set of all finite changes of this choice sequence.

\subsection{Symmetric sequences}\label{section: tools indiscernibility}
Let $\bar{A}=\seqq{A_n}{n<\omega}$ be a countable sequence of sets, in some generic extension of $V$.
For example, we have constructed an interesting sequence $\bar{A}$ in Section~\ref{section: motivation non Pi3}, where we started with a model $V$ of ZFC. 
\begin{remark}
Our base model $V$ is only assumed to be a model of ZF, and \textit{not} choice. This will be important as we iterate the construction.
\end{remark}
In $V(\bar{A})$, let $\P_{\bar{A}}$ be the poset of finite partial choice functions in $\bar{A}$. That is, the poset consisting of finite-support partial functions $p$ with $p(n) \in A_n$ for $n \in \textrm{dom}(p)$, ordered by reverse containment. 
Given a condition $q \in \mathbb{P}_{\bar{A}}$, we say that a finite subset $a \subseteq \omega$ is the support of $q$ if it is the domain of $q$ as a partial function. Note that $\P_{\bar{A}}$ is always bi-definable with $\bar{A}$ over $V$.

Given a subset $a \subseteq \omega$, we write $\bar{A} \upharpoonright a$ for the subsequence $\seqq{A_n}{n\in a}$. Note that for any set $x \in V(\bar{A})$, $x$ is definable from $\bar{A}$ over $V(\bar{A} \upharpoonright a)$ for some finite $a$.

\begin{definition}\label{defn: indiscernible sequences}
We call a sequence $\bar{A}$, appearing in some generic extension of $V$, \textbf{symmetric over $V$} if any two $q_1, q_2 \in \mathbb{P}_{\bar{A}}$ with the same domain $a \subseteq \omega$ have the same type over $V(\bar{A} \upharpoonright (\omega \setminus a))$, $\bar{A}$.
\end{definition}
For example, the sequence from Section~\ref{section: motivation non Pi3} was constructed such that the members of each $A_n$ are indiscernible over $V$. It will follow from Proposition~\ref{prop:technical_tool} below that this sequence is symmetric over $V$ as in Definition~\ref{defn: indiscernible sequences}. First, we show that this definition precisely ensures that the poset $\P_{\bar{A}}$ satisfies the conditions we wanted in Lemma~\ref{lemma: motivation non pot Pi3}.

\begin{proposition}\label{prop:symmetric-cohen-properties}
Suppose $\bar{A}$ is symmetric over $V$. Then the following hold:
\begin{enumerate}
    \item Any choice function for $\bar{A}$ is $\P_{\bar{A}}$-generic over $V(\bar{A})$;
    \item Forcing by $\mathbb{P}_{\bar{A}}$ over $V(\bar{A})$ does not add any new subsets of $V$. In other words, for any choice function $x$, we have $\mathcal{P}(V) \cap V(\bar{A}) = \mathcal{P}(V) \cap V(\bar{A})[x]$; and
    \item If $|A_n| \geq 2$ for infinitely-many $n$, then $V(\bar{A})$ does not have any choice function for $\bar{A}$.
\end{enumerate}
\end{proposition}

\begin{proof}
(1) Fix a choice function $c$ for $\bar{A}$, $c=\seqq{c(n)}{n\in\omega}$, $c(n)\in A_n$. We show that the filter of all $p\in\P_{\bar{B}}$ such that $p=c\restriction\mathrm{dom}(p)$, is generic over $V(\bar{A})$.
Let $D\in V(\bar{A})$ be a dense open subset of $\P_{\bar{A}}$, and fix a finite $a\subset\omega$ such that $D$ is definable from $\bar{A}$ over $V(\bar{A}\restriction a)$.
Since $D$ is dense, there is some $q\in D$ extending $c\restriction a$. Now $c\restriction \mathrm{dom}(q)$ and $q$ have the same type over $V(\bar{A}\restriction a)$, and $D$ is definable using parameters from $V(\bar{A}\restriction a)$, therefore $c\restriction \mathrm{dom}(q)$ is in $D$ as well.

(2) Suppose that $\dot{s}$ is a $\P_{\bar{A}}$-name in $V(\bar{A})$ for a subset of $V$. We can fix some finite $a \subseteq \omega$ and $v \in V$ and formula $\varphi$ such that $\dot{s}$ is definable in $V(\bar{A})$ from $\bar{A}$, $\bar{A}\restriction a$, and $v$ using $\varphi$. 
Let $p$ be the restriction of $x$ to the domain $a$.
We claim that for any two conditions $q_1,q_2\in\P$ extending $p$ and any $z\in V$,
\begin{equation*}
    q_1\Vdash z\in\dot{s}\iff q_2\Vdash z\in\dot{s}.
\end{equation*}
We may assume that $q_1,q_2$ have the same domain $a\cup b$, where $b\cap a=\emptyset$. Let $q_i=p^\frown r_i$ where $r_i\in\P$ is a condition with domain $b$.
The statement $q_i\Vdash z\in\dot{s}$ can be written as a statement about $r_i$ using $\bar{A}$, $v,z$, and parameters in $V(\bar{A}\restriction a)$.
By the indiscernibility assumption in Definition~\ref{defn: indiscernible sequences} we conclude that $r_1$ satisfies this if and only if $r_2$ does.

Finally, $\dot{s}[x]$ can be defined in $V(\bar{A})$ as the set of all $z\in V$ such that there is some condition $q$ extending $p$ (equivalently, any condition $q$ extending $p$) which forces $z\in\dot{s}$.

(3) Suppose for contradiction that $V(\bar{A})$ does have a choice function $c$ for $\bar{A}$. 
Then if $d$ is a $\P_{\bar{A}}$-generic choice function over $V(\bar{A})$, the set $\set{n}{c(n)=d(n)}$ is new to $V(\bar{A})$, contradicting clause (2).
\end{proof}

\begin{proposition}\label{prop:shift_action_ergodic}
Let $\Gamma \le \Delta$ be countable groups, and let $\P$ be the Cohen poset for $2^{\Delta}$, that is, the poset of finite-support functions from $\Delta$ to $2$ ordered by reverse inclusion. Let $a : \Gamma \curvearrowright \P$ be the natural shift action of $\Delta$ on $\P$ restricted to $\Gamma$. If $\Gamma$ is infinite, then the action is ergodic.
\end{proposition}

\begin{proof}
Observe that any condition $p \in \P$ is a partial function $p : \Delta \rightarrow 2$ with finite support, and any two conditions whose supports do not intersect are compatible. Thus it is enough to observe that for any finite subsets $a, b \subseteq \Delta$, there is some $\gamma \in \Gamma$ such that $\gamma a \cap b = \emptyset.$
\end{proof}

The following is easy to verify:

\begin{proposition}\label{prop:step_up}
Suppose for every $n$, $a_n : \Gamma_n \curvearrowright \P_n$ is an ergodic action of a countable group on a poset. Then the natural action
\[ \bigoplus_n \Gamma_n \curvearrowright \prod_n \P_n \]
is ergodic, where $\prod_n \P_n$ is the finite support product of the $\P_n$.
\end{proposition}

Given posets $\mathbb{P}_n$, $n<\omega$, their \textbf{finite support product} is the poset $\mathbb{P}$ of all finite sequences $p$ with $p(i)\in \mathbb{P}_i$ for $i$ in the domain of $p$.
For $p,q\in\mathbb{P}$, $p$ extends $q$ in $\mathbb{P}$ if the domain of $p$ extends the domain of $q$, and $p(i)$ extends $q(i)$ in $\mathbb{P}_i$, for each $i$ in the domain of $q$.

A generic filter $G$ in $\mathbb{P}$ gives rise to a sequence $\seqq{G_n}{n<\omega}$ of filters such that $G_n$ is a generic filter in $\mathbb{P}_n$ (the projections of $G$ to the n'th coordinate). See \cite{Jech2003} or \cite{Kunen2011} for more details.
Below we call the generic filters $x_n$, to emphasize the usual identification of a generic filter with a particular generic object in the generic extension, as is done in the applications of the following proposition.

\begin{proposition}\label{prop:technical_tool}
Suppose that there is a sequence $\P_n$ of posets and commuting actions $\Gamma_n \curvearrowright \P_n$ and $\Sigma_n \curvearrowright \P_n$ in $V$ such that the actions $\Gamma_n \curvearrowright \P_n$ are ergodic. Fix a sequence $\seqq{x_n}{n<\omega}$ which is $\prod_{n< \omega} \P_n$-generic over $V$, define $A_n = \Gamma_n \cdot x_n$, and $\bar{A} := \seqq{\Sigma_n \cdot A_n}{n<\omega }$. Then $\bar{A}$ is symmetric over $V$.
\end{proposition}

\begin{proof}
Suppose that $p$ and $p'$ are two conditions in $\mathbb{P}_{\bar{A}}$ with the same domain $a \subseteq \omega$. 
Fix $\sigma, \sigma' \in \prod_{n \in a} \Sigma_n$ such that $p = \seqq{\sigma_n \cdot (\Gamma_n \cdot x_n)}{n \in a}$ and $p' = \seqq{\sigma_n' \cdot (\Gamma_n \cdot x_n)}{n \in a}$. Let $y = \seqq{\sigma_n \cdot x_n}{n \in a}$ and $y' = \seqq{\sigma_n' \cdot x}{n \in a}$. 

Note that $y$ and $y'$ are $\prod_{n \in a} \P_n$-generic over $V[\seqq{x_n}{n \not\in a}]$, and the model $V(\bar{A})$ can be written as the generic extension $V(\bar{A} \upharpoonright (\omega \setminus a))[y]$ or $V(\bar{A} \upharpoonright (\omega \setminus a))[y']$.
Furthermore, the action of $\prod_{n\in a}\Gamma_n$ on $\prod_{n\in a}\P_n$ is ergodic, and $p,p'$ are the orbits of $y,y'$, respectively, under the action of $\prod_{n\in a}\Gamma_n$.

It follows from Lemma~\ref{lem;gen-orbit-type} that $p$ and $p'$ have the same type over ${V(\bar{A} \upharpoonright (\omega \setminus a)), \bar{A}}$.
\end{proof}

\subsection{Higher rank sets}\label{section:tools higher rank sets}
In our context (forcing over models in which choice fails), not adding reals often coincides with not adding any subsets of ordinals. Similarly, to get irreducibility to the higher Friedman-Stanley jumps, we want to stabilize higher rank sets.

For an ordinal $\theta$ define $\mathcal{P}^0(\theta)=\theta$ and $\mathcal{P}^{n+1}(\theta)=\mathcal{P}(\mathcal{P}^n(\theta))$.
Say that a set $X$ is of \textbf{rank $n$} if it is in $\mathcal{P}^n(\theta)$ for some ordinal $\theta$.
For us the ordinal $\theta$ will usually be $\omega$, in which case sets of rank $0$ are natural numbers, sets of rank $1$ are reals, sets of rank $2$ are sets of reals, and so on.
\begin{lemma}[Monro \cite{Monro1973}]\label{lem;monro-step}
Suppose $M\subseteq N\subseteq K$ are transitive class ZF models such that $M$ and $N$ agree on sets of rank $n$ and $N$ and $K$ agree on all subsets of $M$.
Then $N$ and $K$ agree on all sets of rank $n+1$.
\end{lemma}
\begin{proof}
Prove by induction on $i\leq n+1$ that $N$ and $K$ agree on sets of rank $i$.
For $i=0$ it follows from the assumption, as any ordinal is contained in $M$.
Assume $i\leq n$ and $N,K$ agree on sets of rank $i$.
If $X\in K$ is of rank $i+1$ then the members of $X$ are of rank $i$, and are therefore in $M$. So $X\subseteq M$, and therefore $X\in N$ by assumption.
\end{proof}
In the context of Proposition~\ref{prop:technical_tool}, if $V$ and $V(\bar{A})$ agree on sets of rank $n$, then $V(\bar{A})$ and $V(\seqq{A_n}{n<\omega})$ agree on sets of rank $n+1$.

\section{Lower bounds}\label{sec;lower-bounds}

Our ultimate goal for this section will be to find an action of a tame abelian product group $a\colon \prod_{n\in\omega}\Gamma_n\curvearrowright X$ so that the corresponding orbit equivalence relation $E_a$ is not Borel reducible to $=^{+++}$. By the correspondence in Theorem~\ref{thm: HKL max}, this would imply the desired lower bound stated in Theorem~\ref{thm;main}.
This will be done in Section~\ref{subsection: non pi5}.

First, in Section~\ref{subsec;non-pi3}, we consider tame abelian product groups of the form $\Delta\times\Gamma^\N$, where $\Delta$ is either $\mathbb{Z}_p^{<\omega}$, $\mathbb{Z}$, or $\bigoplus_{\textrm{p prime}} \mathbb{Z}_p$ and $\Gamma$ is infinite, and present actions of these groups which are not Borel reducible to $=_\R^+$, and therefore not potentially $\mathbf{\Pi}^0_3$.
In Section~\ref{subsec: non Pi4} we also provide examples of tame abelian product groups with actions that are not potentially $\mathbf{\Pi}^0_4$.
These examples are optimal by the upper bounds in Section~\ref{section: upper bounds}.

In Section~\ref{subsection:CC jumps}, we will present an application of these methods to recover a result that $E_0^{[\mathbb{Z}^2]}$ is not potentially $\BPi^0_3$.

\subsection{Non potentially $\mathbf{\Pi}^0_3$ actions}\label{subsec;non-pi3}
We are now ready to show the following:
\begin{theorem}
For any countable infinite group $\Gamma$ the following groups all have actions which are not potentially $\mathbf{\Pi}^0_3$:
\begin{enumerate}
    \item $\mathbb{Z}_p^{<\omega}\times\Gamma^\omega$, for any prime $p$,
    \item $(\bigoplus_{\textrm{p prime}}\mathbb{Z}_p)\times\Gamma^\omega$, and
    \item $\mathbb{Z}\times\Gamma^\omega$.
\end{enumerate}  
\end{theorem}
Recall from Section~\ref{section: motivation non Pi3} the action $\alpha_p$ of $\Z_p^{<\omega}\times\Gamma^\omega$ on $X_p^\omega$, and the induced orbit equivalence relation $E_{\alpha_p}$, where  $X_p=(2^\Gamma)^p$.
Part (1) above then follows from Proposition~\ref{prop:motivation non-Pi3}. 
By Propositions \ref{prop:shift_action_ergodic} and \ref{prop:technical_tool}, we conclude that the sequence $\bar{A}$ in Section~\ref{section: motivation non Pi3} is symmetric.
Lemma~\ref{lemma: motivation non pot Pi3} then follows from Proposition~\ref{prop:symmetric-cohen-properties}, which concludes Proposition~\ref{prop:motivation non-Pi3}. 


We now establish parts (2) and (3).
As before, $\Gamma\curvearrowright X_p$ acts diagonally and $\mathbb{Z}_p$ acts on $X_p$ by rotations.
These actions commute, giving an action of $\mathbb{Z}_p\times \Gamma \curvearrowright X_p$.
Identifying $\bigoplus_{p \textrm{ prime}} \mathbb{Z}_p$ as a subgroup of $\prod_p \mathbb{Z}_p$ as before, we get the natural componentwise action
\[ \bigoplus_{p \textrm{ prime}} \Z_p \curvearrowright \prod_{p \textrm{ prime}} X_p\]
which commutes with the componentwise action
\[ \Gamma^{\omega} \curvearrowright \prod_{p \textrm{ prime}} X_p\]
and thus we get the action
\[\alpha_{\oplus} : \bigoplus_{p\textrm{ prime}} \mathbb{Z}_p\times\Gamma^\omega \curvearrowright \prod_p X_p.\]

Considering the natural surjection of $\mathbb{Z}$ onto $\mathbb{Z}_p$, we get the associated actions $\mathbb{Z} \curvearrowright X_p$, from which we can define the diagonal action $\mathbb{Z} \curvearrowright \prod_p X_p$, rotating all coordinates simultaneously. This action commutes with the action $\Gamma^\omega \curvearrowright \prod_p X_p$, so we can also define
\[\alpha_{\mathbb{Z}} : \mathbb{Z}\times\Gamma^\omega \curvearrowright \prod_p X_p.\]

\begin{proposition}\label{proposition: non Pi3 actions of Z and direct sum}
The orbit equivalence relations $E_{\alpha_\oplus}$ and $E_{\alpha_\mathbb{Z}}$ are not Borel reducible to $=^+$.
\end{proposition}
\begin{proof}
Let $\mathbb{Q}_p$ be Cohen forcing on $X_p$.
Let $\mathbb{Q}$ be the finite support product of $\mathbb{Q}_p$, $\seqq{x_p}{p\textrm{ prime}}$ a $\mathbb{Q}$-generic over $V$, $A_p=\Gamma\cdot x_p$ and $\bar{A}=\seqq{\Z_p \cdot A_p}{p\textrm{ prime}}$.
By Proposition \ref{prop:technical_tool}, the sequence $\bar{A}$ is symmetric over $V$. Thus forcing with $\P_{\bar{A}}$ over $V(\bar{A})$ adds no new reals by Proposition \ref{prop:symmetric-cohen-properties}.

As before $\hat{A}=\bigoplus_p\mathbb{Z}_p\cdot \seqq{A_n}{n<\omega}$ is an absolute invariant for $E_{\alpha_\oplus}$, and by Example~\ref{ex;finite-choice}, the action of $\bigoplus_p \mathbb{Z}_p$ on $\mathbb{P}_{\bar{A}}$ is ergodic. As in Corollary~\ref{cor;irred-to-jumps}, it follows from Lemma~\ref{lem;weak-hom} that if $B$ is a set of reals definable from $\bigoplus_p\mathbb{Z}_p\cdot \seqq{A_n}{n<\omega}$ over $V(\bar{A})$, then $B$ in the base model $V(\bar{A})$. In particular, $V(B)\subset V(\bar{A})$ is strictly smaller than $V(\hat{A})$. It follows from Corollary~\ref{cor : irred to cong} (1) that $E_{\alpha_{\oplus}} \not\le_B =^+$.
    
Similarly, $\mathbb{Z}\cdot \seqq{A_n}{n<\omega}$ is a generic invariant for $E_{\alpha_{\mathbb{Z}}}$.
$\mathbb{Z}$ acts on the poset $\mathbb{P}_{\bar{A}}$ by rotating all coordinates simultanously.
This action is ergodic by Example~\ref{example: chinese remainder}.
The rest of the proof is similar to the above.
\end{proof}

In the next stage of our construction we actually need a more uniform version of what we proved here.
We wish to find a sequence $\seqq{A_{n, p}}{p \textrm{ prime},\; n < \omega}$ 
such that each $\mathbb{Z}_p^{<\omega} \cdot \seqq{A_{n, p}}{n < \omega}$ is an absolute invariant for $E_{\alpha_p}$, 
and the sequence $\bar{A} = \seqq{\Z_p \cdot A_{n, p}}{p \textrm{ prime},\; n < \omega}$ is symmetric over $V$.

To do this, for each prime $p$, let $\mathbb{Q}_p$ the poset to add a Cohen-generic element of $X_p^\omega$, which we can write as the finite-support product of $\mathbb{Q}_{p, n}$, where each $\mathbb{Q}_{p, n}$ adds a Cohen-generic element of $X_p$. 

Force with $\prod_{p \textrm{ prime}} \Q_p$ over $V$ to get a generic sequence $\seqq{x_{p, n}}{p \textrm{ prime},\; n < \omega}$. Define $A_{p, n} = \Gamma \cdot x_n$ for $n < \omega$ and prime $p$, and $\bar{A} = \seqq{\Z_p \cdot A_{p, n}}{p \textrm{ prime},\; n < \omega}$.
It follows from Propositions \ref{prop:shift_action_ergodic} and \ref{prop:technical_tool} that $\bar{A}$ is symmetric over $V$.

\subsection{Non potentially $\BPi^0_4$ actions}\label{subsec: non Pi4}
In this section we show
\begin{theorem}\label{thm: non Pi4 subsection}
For any countable infinite group $\Gamma$ the following groups have actions which are not potentially $\mathbf{\Pi}^0_4$:
\begin{enumerate}
    \item $\bigoplus_{p\textrm{ prime}}\mathbb{Z}_p\times\prod_{p\textrm{ prime}}\mathbb{Z}_p^{<\omega}\times\Gamma^\omega$;
    \item $\Z\times\prod_{p\textrm{ prime}}\mathbb{Z}_p^{<\omega}\times\Gamma^\omega$.
\end{enumerate}  
\end{theorem}

For every prime $p$, let $X_p = (2^\Gamma)^p$, and define $\alpha_p$ to be the action
\[ \alpha_p: \mathbb{Z}_p^{<\omega} \times \Gamma^\omega \curvearrowright Y_p\]
as before, where $Y_p = X_p^\omega$. We saw that $E_{\alpha_p}$ is not Borel-reducible to $=^+$ for every prime $p$, and that 
\[ x \mapsto \mathbb{Z}_p^{<\omega} \cdot \seqq{\Gamma \cdot x_n}{n < \omega}\]
is an absolute classification of $E_{\alpha_p}$.

For every prime $p$, the group $\mathbb{Z}_p$ acts on $Y_p$ with the diagonal action,
\[ \sigma \cdot \seqq{x_n}{n < \omega} = \seqq{\sigma \cdot x_n}{n < \omega},\]
and this action commutes with $\alpha_p$. Thus the coordinate-wise actions 
\[\bigoplus_{p \textrm{ prime}}\mathbb{Z}_p \curvearrowright \prod_{p \textrm{ prime}} Y_p,\]
and
\[\alpha_{\prod p} : \prod_{p \textrm{ prime}} (\mathbb{Z}_p^{<\omega} \times \Gamma^\omega) \curvearrowright \prod_{p \textrm{ prime}} Y_p\]
commute as well.

As in Section~\ref{section: tools actions on orbits}, we get an action
\[ \beta_{\oplus} : (\bigoplus_{p \textrm{ prime}} \mathbb{Z}_p) \times \prod_{p \textrm{ prime}} (\mathbb{Z}_p^{<\omega} \times \Gamma^\omega) \curvearrowright \prod_{p \textrm{ prime}} Y_p.\]

By considering the natural surjections of $\mathbb{Z}$ onto $\mathbb{Z}_p$, we can as before consider the diagonal action $\mathbb{Z} \curvearrowright \prod_{p \textrm{ prime}} Y_p$ which also commutes with the action $\alpha_{\prod p}$ and thus we can also define the action

\[ \beta_{\mathbb{Z}} : \mathbb{Z} \times \prod_{p\textrm{ prime}} (\mathbb{Z}_p^{<\omega} \times \Gamma^\omega) \curvearrowright \prod_{p \textrm{ prime}} Y_p.\]
Note that the groups acting in $\beta_{\oplus}$ and $\beta_\Z$ respectively are isomorphic to the group in Theorem~\ref{thm: non Pi4 subsection} parts (1) and (2) respectively.
\begin{proposition}
The orbit equivalence relations $E_{\beta_{\oplus}}$ and $E_{\beta_{\mathbb{Z}}}$ are not Borel-reducible to $=^{++}$.
\end{proposition}

\begin{proof}
We will prove this result just for $E_{\beta_{\oplus}}$ since the argument for $E_{\beta_{\mathbb{Z}}}$ is analogous as in the previous section.

We established at the end of the previous section that there is a sequence of $A_{n, p}$ for prime $p$ and $n < \omega$
such that each $\mathbb{Z}_p^{<\omega} \cdot \seqq{A_{n, p}}{n < \omega}$ is an absolute invariant for $E_{\alpha_p}$, 
and the sequence $\bar{A} = \seqq{\Z_p \cdot A_{n, p}}{p \textrm{ prime},\; n < \omega}$ is symmetric over $V$.
We now define $B_p$ to be $\Z_p^{<\omega} \cdot \seqq{A_{n, p}}{n < \omega}$ for each prime $p$.
Observe that $\seqq{B_p}{p \textrm{ prime}}$ is an absolute invariant for $\alpha_{\prod p}$, and
$\hat{B} = \bigoplus_{p \textrm{ prime}} \Z_p \cdot \seqq{B_p}{p \textrm{ prime}}$ is an absolute invariant for $\beta_{\oplus}$.
Now define $\bar{B}$ to be $\seqq{\Z_p \cdot B_p}{p \textrm{ prime}}$.

\begin{claim}\label{claim:symmetric2}
The sequence $\bar{B}$ is symmetric over $V(\bar{A})$.
\end{claim}
\begin{proof}
The claim is a direct application of Proposition \ref{prop:technical_tool} as follows.
Let $\bar{A}_p=\seqq{\Z_p\cdot A_{n,p}}{n<\omega}$, and $c_p=\seqq{A_{n,p}}{n<\omega}$. Then $\seqq{c_p}{p\textrm{ prime}}$ is generic for $\prod_{p\textrm{ prime}}\P_{\bar{A}_p}$. 
$\Gamma_p=\Z_p^{<\omega}$ acts ergodically on $\P_{\bar{A}_p}$, by Example~\ref{ex;finite-choice}. Furthermore, the actions of $\Sigma_p=\Z_p$ on $\P_{\bar{A}_p}$ commutes with the action of $\Gamma_p$. 
\end{proof}
Note that the model generated by our invariant $\hat{B}$, $V(\hat{B})$, is equal to $V(\seqq{B_p}{p \textrm{ prime}})$, and this is a strict extension of $V(\bar{B})$, since $V(\bar{B})$ does not have a choice function through $\bar{B}$.

\begin{claim}\label{claim:same_sets_of_reals}
The models $V(\bar{B})$ and $V(\seqq{B_p}{p \textrm{ prime}})$ have the same sets of reals.
\end{claim}

\begin{proof}
Recall that $V(\bar{A})$ and $V(\seqq{A_{n, p}}{p \textrm{ prime},\; n < \omega})$ have the same reals. Since $\bar{B}$ is symmetric over $V(\bar{A})$, Proposition~\ref{prop:symmetric-cohen-properties} implies that $V(\bar{B})$ and $V(\seqq{B_p}{p \textrm{ prime}})$ have the same subsets of $V(\bar{A})$. By Lemma~\ref{lem;monro-step} they have the same sets of reals.
\end{proof}

Now suppose for the sake of contradiction that there is a Borel reduction from $E_{\beta_\oplus}$ to $=^{++}$. 
By Corollary~\ref{cor : irred to cong} (2), there is a set of sets of reals $I$ such that $V(\hat{B}) = V(I)$ and $I$ is definable from $\hat{B}$ over $V$.
by Claim~\ref{claim:same_sets_of_reals}, $I\subset V(\bar{B})$.
By Example~\ref{ex;finite-choice}, the action $\bigoplus_{p \textrm{ prime}} \Z_p \curvearrowright \P_{\bar{B}}$
is ergodic.
It follows from Lemma \ref{lem;weak-hom} that $I \in V(\bar{B})$. In particular, $V(I)$ is contained in $V(\bar{B})$, and therefore is not equal to $V(\hat{B})$, a contradiction.
\end{proof}

\begin{remark}\label{remark: restriction to primes in P}
Suppose $P$ is an infinite set of primes. By following the above, but ranging $p$ only over primes in $P$, we get actions
\[ \beta_{\oplus P} : \bigoplus_{p \in P} \mathbb{Z}_p \times \prod_{p \in P} (\mathbb{Z}_p^{<\omega} \times \Gamma^\omega) \curvearrowright \prod_{p \in P} Y_p\]
and
\[ \mathbb{Z} \times \prod_{p \in P} (\mathbb{Z}_p^{<\omega} \times \Gamma^\omega) \curvearrowright \prod_{p \in P} Y_p\]
inducing orbit equivalence relations which are not Borel reducible to $=^{++}$.
These will be used in the following section.
\end{remark}

\subsection{Non potentially $\BPi^0_5$-actions}\label{subsection: non pi5}
Fix a partition $P_0,P_1,P_2,...$ of the prime numbers into infinite sets.
\begin{theorem}\label{thm: non Pi5 subsection}
For any countable infinite group $\Gamma$ the group
\begin{equation*}
    \mathbb{Z}^{<\omega}\times\prod_n\bigoplus_{p\in P_n}\mathbb{Z}_p\times\prod_p\mathbb{Z}_p^{<\omega}\times\Gamma^\omega
\end{equation*}
has an action which is not potentially $\mathbf{\Pi}^0_5$.
\end{theorem}
If, for example, $\Gamma=\mathbb{Z}(q^\infty)$ for some prime $q$, then the above group is tame according to Solecki's characterization. We therefore conclude the lower bound in Theorem~\ref{thm;main}.

In the previous section we considered the diagonal action $\mathbb{Z}_p \curvearrowright Y_p$ defined by
\[ \sigma \cdot \seqq{x_n}{n < \omega} = \seqq{\sigma \cdot x_n}{n < \omega}, \]
which we extend to an action of $\mathbb{Z}$ by the natural surjection of $\mathbb{Z}$ onto $\mathbb{Z}_p$.
Now observe that there is a diagonal action
\[ \mathbb{Z} \curvearrowright \prod_{p \in P_n} Y_p \]
defined in terms of these other diagonal actions by
\[ z \cdot \seqq{\seqq{x_{p, n}}{p \in P_n}}{n < \omega} = \seqq{z \cdot \seqq{x_{p, n}}{p \in P_n}}{n < \omega}.\]
This action commutes with $\beta_{\oplus P_n}$ from Remark~\ref{remark: restriction to primes in P}. 

Considering the componentwise action
\[ \mathbb{Z}^{<\omega} \curvearrowright \prod_{n < \omega} \prod_{p \in P_n} Y_p\]
and the componentwise action of the product of $\beta_{\oplus P_n}$'s, we get an action
\[\beta_{\pi} : \prod_{n < \omega} \left[ \bigoplus_{p \in P_n} \mathbb{Z}_p \times \prod_{p \in P_n} (\mathbb{Z}_p^{<\omega} \times \Gamma^\omega) \right] \curvearrowright \prod_{n < \omega} \prod_{p \in P_n} Y_p.\]
These two actions commute, providing the action
\[\gamma : \mathbb{Z}^{<\omega} \times\prod_{n < \omega} \left[ \bigoplus_{p \in P_n} \mathbb{Z}_p \times \prod_{p \in P_n} (\mathbb{Z}_p^{<\omega} \times \Gamma^\omega) \right] \curvearrowright \prod_{n < \omega} \prod_{p \in P_n} Y_p.\]
Note the the group in the action $\gamma$ is isomorphic to the group in Theorem~\ref{thm: non Pi5 subsection}.
\begin{proposition}
The orbit equivalence relation $E_\gamma$ is not Borel-reducible to $=^{+++}$.
\end{proposition}

\begin{proof}
Let us denote the ground model of ZFC, which we previously denoted $V$, as $V_0$. We next use $V$ to denote the model we previously denoted $V(\bar{A})$ in Section~\ref{subsec: non Pi4} above. That is, $V$ is a model of ZF, and furthermore, we have a sequence $\seqq{B_p}{p \textrm{ prime}}$, in a generic extension of $V$, which serves as an absolute invariant for $\beta_{\oplus}$ such that $V$ and $V(\bar{B})$ share the same sets of reals, where $\bar{B} = \seqq{\Z_p \cdot B_p}{p \textrm{ prime}}$.
Note that $V(\bar{B})$ is equal to $V_0(\bar{B})$, as the set $\bar{A}$ can be defined from $\bar{B}$.

For any $n$, $\bigoplus_{p \in P_n}\Z_p \cdot \seqq{B_p}{p\in P_n}$ is an absolute invariant for $\beta_{\oplus P_n}$.
We define $C_n$ to be this set, so that $\seqq{C_n}{n<\omega}$ is an absolute invariant for $\beta_{\pi}$. 
Observe that $\mathbb{Z}^{<\omega} \cdot \seqq{C_n}{n < \omega}$ is an absolute invariant for $\gamma$. 
Finally, define $\bar{C}= \seqq{\mathbb{Z} \cdot C_n}{n < \omega}$.

\begin{claim}
The sequence $\bar{C}$ is symmetric over $V(\bar{B})$.
\end{claim}
\begin{proof}
The claim follows from Proposition \ref{prop:technical_tool}. The argument is similar to Claim~\ref{claim:symmetric2}.
\end{proof}
Note that $V_0(\hat{C})$ is equal to the model $V_0(\seqq{C_n}{n < \omega})$ and also to $V(\hat{C})$, or $V(\seqq{C_n}{n < \omega})$. Furthermore, this model strictly extends $V(\bar{C})$, as the latter does not contain the sequence $\seqq{C_n}{n < \omega}$, which is a choice sequence through $\bar{C}$.

\begin{claim}\label{claim:same_sets_of_sets_of_reals}
The models $V(\bar{C})$ and $V(\seqq{C_n}{n<\omega})$ have the same sets of sets of reals.
\end{claim}
\begin{proof}
Recall that $V(\bar{B})$ and $V(\seqq{B_p}{p \textrm{ prime}})$ have the same sets of reals.  Because $\bar{C}$ is symmetric over $V(\bar{B})$, we know by Proposition~\ref{prop:symmetric-cohen-properties} that $V(\bar{C})$ and $V(\seqq{C_n}{n<\omega})$ have the same subsets of $V(\bar{B})$. Thus by Lemma~\ref{lem;monro-step} they have the same sets of sets of reals.
\end{proof}

Now we assume for the sake of contradiction that $E_\gamma$ is Borel reducible to $=^{+++}$. 
By Corollary~\ref{cor : irred to cong} (3), there is a set of sets of sets of reals $I$ such that $V(\hat{C}) = V(I)$ and $I$ is definable from $\hat{C}$ over $V$.
By Claim~\ref{claim:same_sets_of_sets_of_reals}, $I\subset V(\bar{C})$.
By Example~\ref{ex;finite-choice}, the action $\Z^{<\omega} \curvearrowright \P_{\bar{C}}$
is ergodic.
It follows from Lemma \ref{lem;weak-hom} that $I \in V(\bar{C})$. In particular, $V(I)$ is contained in $V(\bar{C})$, and therefore is not equal to $V(\hat{C})$, a contradiction.
\end{proof}

\section{An application to the Clemens-Coskey jumps}\label{subsection:CC jumps}

Recall the definitions of the $\Gamma$-jumps of Clemens and Coskey from Section~\ref{sec: intro CC jumps}.
In this section we prove Theorem~\ref{thm: Z2 jump non pi3}, that the $\mathbb{Z}^2$-jump of $E_0$, $E_0^{[\mathbb{Z}^2]}$, is not Borel reducible to $=^+$. 

As mentioned in Section~\ref{sec: intro CC jumps}, Clemens and Coskey showed that, when restricted to a comeager subset of its domain, the equivalence relation $E_0^{[\mathbb{Z}^2]}$ is in fact Borel reducible to $=^+$. The main point here is to restrict the domain of $E_0^{[\mathbb{Z}^2]}$ to a certain meager set, in which the complexity hides. 
This set will be those elements in $(2^\omega)^{\mathbb{Z}^2}$ such that each row is periodic, with a distinct prime period.
The non-reducibility result then follows as in the second part of Proposition~\ref{proposition: non Pi3 actions of Z and direct sum}, based on the Chinese-remainder theorem.
In fact, we show here that the equivalence relation $E_{\alpha_{\Z}}$, which is not Borel reducible to $=^+$, is Borel reducible to $E_0^{[\Z^2]}$.

Recall the action $\alpha_\Z$ of $\mathbb{Z}\times\Gamma^\omega$ on $\prod_p(2^\Gamma)^p$ described before Proposition~\ref{proposition: non Pi3 actions of Z and direct sum}, and let $E_{\alpha_\Z}$ be the orbit equivalence relation.
Take $\Gamma$ to be $\mathbb{Z}$.
Note that $(E_0^\omega)^{[\mathbb{Z}]}\leq_B E_0^{[\mathbb{Z}^2]}$.
The following lemma then concludes the proof of Theorem~\ref{thm: Z2 jump non pi3}.
\begin{lemma}
$E_{\alpha_\Z}$ is Borel reducible to $(E_0^\omega)^{[\mathbb{Z}]}$.
\end{lemma}
\begin{proof}
$E_0^\omega$ is defined on $(2^\omega)^\omega$ by $a\mathrel{E_0^\omega} b\iff(\forall n<\omega)a(n)\mathrel{E_0}b(n)$.  $(E_0^\omega)^{[\mathbb{Z}]}$ is defined on $((2^\omega)^\omega)^{\mathbb{Z}}$ by 
\begin{equation*}
    x\mathrel{(E_0^\omega)^{[\mathbb{Z}]}}y\iff (\exists k\in\mathbb{Z})(\forall l\in\mathbb{Z}) x(l+k)\mathrel{E_0^\omega}y(l).
\end{equation*}

Let $F_p$ be the orbit equivalence relation induced by the action of $\Gamma$ on $X^p=(2^\Gamma)^p$. Since this is a generically ergodic action of $\Gamma=\mathbb{Z}$, $F_p$ is Borel bireducible with $E_0$ for all $p$ (see~\cite[Theorem 6.6]{KechrisMiller2004}).
Let $P\subset\omega$ be the set of prime numbers.
Let $F=\prod_{p\in P}F_p$ be the product relation defined on the Polish space $\prod_{p\in P}X^p$. Then $F$ is Borel bireducible with $E_0^\omega$, and $F^{[\mathbb{Z}]}$ is Borel bireducible with $(E_0^\omega)^{[\mathbb{Z}]}$.
We conclude the lemma by showing that $E_{\alpha_\Z}$ is Borel reducible to $F^{[\mathbb{Z}]}$.

We view the space $\prod_p (2^\Gamma)^p$ as $P$-many (vertical) copies of the Polish spaces $(2^\Gamma)^p$, and we view $(\prod_p (2^\Gamma)^p)^{\mathbb{Z}}$ as $\mathbb{Z}$-many (horizontal) copies of $\prod_p (2^\Gamma)^p$.

Given $x\in \prod_{p}(2^\Gamma)^p$, define $f(x)\in (\prod_{p}(2^\Gamma)^p)^{\mathbb{Z}}$ as follows.
\begin{equation*}
f(x)(l)(p)=(l \mod p)\cdot x(p).    
\end{equation*}
That is, for a given $p\in P$, the (horizontal) $\mathbb{Z}$-line $\seqq{f(x)(l)(p)}{l\in\mathbb{Z}}$ is periodic, with period $p$, and it takes the values in the orbit $\mathbb{Z}_p\cdot x(p)$, $x(p), 1\cdot x(p),..., (p-1)\cdot x(p)$. We show that $f$ is a Borel reduction of $E_{\alpha_\Z}$ to $F^{[\mathbb{Z}]}$.

Fix $x\in \prod_{p}(2^\Gamma)^p$ and $\gamma\in\prod_p\Gamma$. Then $f((0,\gamma)\cdot x)(l)(p)$ and  $f((0,\gamma)\cdot x)$ are $F$-related, and therefore are $F^{\mathbb{[Z]}}$-related. Let $\bar{1}$ be the identity element in the group $\prod_p\Gamma$.
For $k\in\mathbb{Z}$, $f((k,\bar{1})\cdot x)(l)(p)=f(x)(p)(l-k)$. Thus $f(x)$ and $f((k,\bar{1})\cdot x)$ are $(\prod_p F_p)^{[\mathbb{Z}]}$-related.
We conclude that $f$ is a homomorphism from $E_{\alpha_\Z}$ to $F^{[\mathbb{Z}]}$.

Assume now that $f(x)$ and $f(y)$ are $F^{[\mathbb{Z}]}$-related. Fix $k$ such that $f(x)(l+k)\mathrel{F}f(y)(l)$ for all $l\in\mathbb{Z}$.
Fix $\gamma\in\prod_p\Gamma$ such that $\gamma(p)\cdot f(x)(k)(p)=f(y)(0)(p)$ for all prime $p$.
Then $(k,\gamma)\cdot x=y$.
We conclude that $f$ is a reduction, as required.
\end{proof}

\section{Upper bounds}\label{section: upper bounds}

In this section we show that a minor modification of the arguments of Ding and Gao in \cite{DingGao2017} give the optimal upper bound of $D(\BPi^0_5)$ for the potential Borel complexity of any orbit equivalence relation induced by an action of a tame abelian product group.

\begin{proposition}[Ding-Gao]\label{pr:ding_gao_upper}
Suppose that $\prod_{n \in \omega} \Gamma_n$ is a tame abelian product group such that $\Gamma_n$ is torsion for every $n$. Then $\prod_{n \in \omega} \Gamma_n$ is $\BPi^0_5$-tame.

Suppose furthermore that for every $n$, the group $\Gamma_n$ does not have $\mathbb{Z}_p^{<\omega}$ for prime $p$ as a subgroup. Then $\prod_{n \in \omega} \Gamma_n$ is in fact $\BPi^0_4$-tame.
\end{proposition}

\begin{proof}
It is proven in \cite[Theorem 5.15]{DingGao2017} that in the first case any orbit equivalence relation induced by the product group is Borel-reducible to $(E_0^\omega)^{++}$, and in \cite[Theorem 5.14]{DingGao2017} that in the second case it is Borel-reducible to $(E_0^\omega)^{+}$. It is easy to see directly that $E_0^\omega \le_B =^+$ in which case we get Borel reductions to $=^{+++}$ and $=^{++}$ respectively. The fact that these relations are potentially $\BPi^0_5$ and potentially $\BPi^0_4$ respectively can be found in \cite[Theorem 2]{HKL1998}.
\end{proof}

\begin{lemma}\label{lem:d_upper_bound}
Let $G$ be a Polish group and $\Delta$ a countable group, and suppose that $\Delta \times G$ acts continuously on a Polish space $X$. We also let $G$ act on $X$ with the subaction, so that we get orbit equivalence relations $E^{\Delta \times G}_X$ and $E^G_X$. If $E^{G}_X$ is potentially $\BPi^0_\alpha$, then $E^{\Delta \times G}_X$ is potentially $\BSigma^0_{\alpha+1}$.
\end{lemma}

\begin{proof}
Fix a compatible Polish topology $\tau$ on $X$ such that $E^G_X$ is $\BPi^0_\alpha$ as a subset of $X \times X$ in the product toplogy $\tau \times \tau$. By direct computation, $E^{\Delta \times G}$ can be seen to be $\BSigma^0_{\alpha+1}$ as a subset of $X \times X$ in this topology.
\end{proof}

\begin{theorem}
If $\prod_{n \in \omega} \Gamma_n$ is a tame abelian product group, then any orbit equivalence relation induced by an action of $\prod_{n \in \omega} \Gamma_n$ is potentially $D(\BPi^0_5)$.
\end{theorem}

\begin{proof}
Let $X$ be a Polish space with a continuous action of $\prod_{n \in \omega} \Gamma_n$. By Solecki's characterization \cite{Solecki1995}, if the product group is tame then there must be some $n_0$ such that for every $k \ge n_0$, the group $\Gamma_k$ is torsion. Let $\Delta = \prod_{k < n_0} \Gamma_k$ and $G = \prod_{k \ge n_0} \Gamma_k$. 
By Proposition \ref{pr:ding_gao_upper}, the orbit equivalence relation $E^G_X$ is potentially $\BPi^0_5$. 
By Lemma \ref{lem:d_upper_bound}, the orbit equvialence relation $E^{\Delta \times G}_X$ is $\Sigma^0_6$. 
As the group $\Delta \times G = \prod_{n \in \omega} \Gamma_n$ is a non-Archimedean Polish group, by \cite{HKL1998}, the potential complexity is actually $D(\BPi^0_5)$.
\end{proof}

\bibliographystyle{alpha}
\bibliography{bibliography}
\end{document}